\documentclass[english,11pt]{smfart}
  \usepackage{amsmath,amssymb,smfthm}
  \usepackage{mathrsfs}
  \usepackage{graphicx,color}

  \newcommand{\wassd}{\mathop \mathrm{W}\nolimits}
  \newcommand{\mdim}{\mathop \mathrm{mdim}_M\nolimits}
  \newcommand{\entropy}{\mathop \mathrm{h}_{\mathrm{top}}\nolimits}
  \newcommand{\supp}{\mathop \mathrm{supp}\nolimits}

  \newcommand{\hcube}{\mathop \mathrm{HC}\nolimits}
  \newcommand{\bcube}{\mathop \mathrm{BC}\nolimits}
  \newcommand{\diam}{\mathop \mathrm{diam}\nolimits}
  \newcommand{\crit}{\mathop \mathrm{crit}\nolimits}
  \newcommand{\udim}{\mathop {\operatorname{\overline{M}-dim}}\nolimits}
  \newcommand{\ldim}{\mathop {\operatorname{\underline{M}-dim}}\nolimits}
  \newcommand{\ucrit}{\mathop {\operatorname{\overline{M}-crit}}\nolimits}
  \newcommand{\lcrit}{\mathop {\operatorname{\underline{M}-crit}}\nolimits}

  \title[Generalized Hausdorff dimension]{A generalization of Hausdorff 
     dimension applied to Hilbert cubes and Wasserstein spaces}

  \author{Beno\^{\i}t Kloeckner}

\begin{document}

\begin{abstract}
A Wasserstein space is a metric space of sufficiently concentrated
probability measures over a general metric space. The main goal
of this paper is to estimate the largeness of Wasserstein spaces, in a sense to be 
made precise.

In a first part, we generalize the Hausdorff dimension by defining
a family of bi-Lipschitz invariants, called critical parameters,
that measure largeness for 
infinite-dimensional metric spaces. Basic properties of these invariants 
are given, and they are
estimated for a naturel set of spaces generalizing the usual Hilbert cube.
These invariants are very similar to concepts initiated by Rogers,
but our variant is specifically suited to tackle Lipschitz comparison.

In a second part, we estimate the value of these new invariants in the
case of some Wasserstein spaces, as well as the dynamical complexity of 
push-forward maps. The lower bounds rely on several embedding
results; for example we provide uniform bi-Lipschitz embeddings of all powers
of any space inside its Wasserstein space and 
we prove that the Wasserstein space of a $d$-manifold has ``power-exponential''
critical parameter equal to $d$.
These arguments are very easily adapted to study the space of closed subsets of a 
compact metric space, partly generalizing results of Boardman, Goodey and McClure.
\end{abstract}

\maketitle

\section{Introduction}

This article is motivated by the geometric study of Wasserstein spaces;
these are spaces of probability measures over a metric space, which are often
infinite-dimensional for any sensible definition of dimension (in particular
Hausdorff dimension). This statement seemed to deserve to be made quantitative,
and very few relevant invariants seemed available. We shall therefore
develop such tools in a first part, then apply them to Wasserstein
spaces via embedding results in a second part.

\subsection{A generalization of Hausdorff dimension: 
critical parameters}

The construction of Hausdorff dimension relies on a family of functions,
namely $(r\mapsto r^s)_s$, and one can wonder what happens when this family
is replaced by another one. This is exactly what we do: we give conditions
on a family of functions (then called a \emph{scale}) ensuring that a family of 
measures obtained by the so-called Carath\'eodory construction
from these functions behave more or less like Hausdorff 
measures do. In particular
these criterions ensure the existence of a \emph{critical parameter} that plays
the role of Hausdorff dimension, and the Lipschitz invariance of 
this parameter. It follows that any bi-Lipschitz embedding
of a space into another implies an inequality between their critical parameters.
We shall use three main scales relevant for increasingly large spaces:
the \emph{polynomial} scale, which defines the Hausdorff dimension;
the \emph{intermediate} scale and the \emph{power-exponential} scale.
We shall say for example that a space has intermediate size if it has a
non-extremal critical parameter in the intermediate scale, which implies
that it has infinite Hausdorff dimension and minimal critical parameter
in the power-exponential scale.

This line of ideas is far from being new: Rogers' book 
\cite{Rogers} shows that this kind of constructions were well understood
forty years ago. Several works have considered
infinite-dimensional metric spaces, mostly the set of
closed subsets of the interval, and determined for some functions
whether they lead to zero or infinite measures; see in particular 
\cite{Boardman,Goodey2,McClure}.

Concerning the definition
of critical parameters, our main point is to stress conditions ensuring
their bi-Lipschitz invariance. But the real contribution of this paper lies
in the computation of critical parameters for a variety of spaces,
partly generalizing the above papers.

Hausdorff dimension is easy to interpret because the Eulidean spaces can
be used for size comparison. There is a natural family of spaces that
can play the same role for some families of critical parameter: Hilbert cubes.
Given an $\ell^2$ sequence of positive real numbers 
$\bar a=(a_n)_{n\in \mathbb{N}}$ (the classical choice being $a_n=1/n$),
let $\hcube(I;\bar a)$ be the set of all sequences $\bar u$ such that 
$0\leqslant u_n\leqslant a_n$ for all $n$, and endow it with the $\ell^2$ metric. 
Here $I$ stands for the unit interval, and the construction generalizes
to any compact metric space $X$: the (generalized) Hilbert cube 
$\hcube(X;\bar a)$ is the set of sequences $\bar x=(x_n)\in X^{\mathbb{N}}$
endowed with the metric
\[d_{\bar a}(\bar x,\bar y) := \left(\sum_{n=1}^\infty a_n^2 d(x_n,y_n)^2\right)^{1/2}\] 

The main results of the first part are estimations of the critical parameters
of generalized Hilbert cubes. In particular, we prove that under
positive and finite dimensionality hypotheses, $\hcube(X,\bar a)$
has intermediate size if $\bar a$ decays exponentially, and has power-exponential
size if $\bar a$ decays polynomially. 

To illustrate this, let us give a consequence of our estimations.
\begin{coro}\label{coro:application}
Let $X,Y$ be any two compact metric spaces, assume $X$ has
positive Hausdorff dimension and $Y$ has finite upper Minkowski dimension,
and consider two exponents $\alpha<\beta\in(1/2,+\infty)$.
 
Then there is no bi-Lipschitz embedding
$\hcube(X;(n^{-\alpha})) \hookrightarrow \hcube(Y;(n^{-\beta}))$.
\end{coro}

This non-embedding result, as well as a similar result described
below, is different in nature to the celebrated results of
Bourgain \cite{Bourgain} (a regular tree admits no bi-Lipschitz
embedding into a Hilbert space), Pansu \cite{Pansu} and
Cheeger and Kleiner \cite{Cheeger-Kleiner} (the Heisenberg group
admits no bi-Lipschitz into a finite-dimensional Banach space nor into $L^1$).
These results involve the fine structure of metric spaces, while our approach
is much cruder: all our non-embedding results come from one space being 
simply too big to fit into another.

Our methods are similar to those used in Hausdorff dimension
theory: we rely on Frostman's Lemma, which says that in order to bound from 
below the critical parameter it is sufficient to exhibit a measure whose local 
behavior is controlled by one of the scale functions, and on an analogue of 
Minkowski dimension, which gives upper bounds.

This analogue might be considered the most straightforward manner to measure
the largeness of a compact space: it simply encodes
the asymptotics of the minimal size of 
an $\varepsilon$-covering when $\varepsilon$ go to zero.
However, the Minkowski dimension has some undesirable behavior, notably with 
respect to
countable unions; this already makes Hausdorff dimension more satisfactory,
and the same argument applies in favor of our critical parameters.

Whatever scale is used, the construction of critical parameters relies on
the existence for all $\varepsilon>0$ of at least one covering of the space by a
sequence of parts $E_n$ whose diameter is at most $\varepsilon$ and goes to
$0$ when $n$ goes to $\infty$. This property has been studied under the names 
``small ball property'' and ``largest Hausdorff dimension'', see
the works of Goodey \cite{Goodey}, Bandt \cite{Bandt} and
Behrends and Kadets \cite{Behrends-Kadets}. In particular, it is proved
in \cite{Goodey} and \cite{Behrends-Kadets}
that the unit ball of an infinite-dimensional Banach space never has the 
small ball property. As a consequence, our critical parameters cannot be
used to measure the largeness of Banach spaces, apart from the obvious
relation between Hausdorff dimension and linear dimension of finite-dimensional
Banach spaces.

\subsection{Largeness of Wasserstein spaces}

The second part of this article is part of a series, 
partly joint with J\'er\^ome Bertrand,
in which we study some intrinsic geometric properties
of the \emph{Wasserstein spaces} $\mathscr{W}_p(X)$ of a metric space
$(X,d)$. These spaces of measures
are in some sense geometric measure theory versions of $L^p$ spaces 
(see Section \ref{sec:recalls} for precise definitions).
Here we evaluate the largeness of Wasserstein spaces, mostly
via embedding results.

Other authors have worked on related topics, for example
Lott \cite{Lott}, who computed the curvature of Wasserstein spaces over 
manifolds (see also Takatsu \cite{Takatsu}), and Takatsu and Yokota
\cite{Takatsu-Yokota} who studied the case when $X$ is a metric cone.

Several embedding and non-embedding results are proved in previous articles
for special classes of spaces $X$, in the most important case $p=2$.
On the first hand, it is easy to see that if $X$ contains a complete geodesic
(that is, an isometric embedding of $\mathbb{R}$),
then $\mathscr{W}_2(X)$ contains isometric embeddings of
open Euclidean cone of arbitrary dimension \cite{Kloeckner}.
In particular it contains isometric
embeddings of Euclidean balls of arbitrary dimension and radius,
and bi-Lipschitz embeddings of $\mathbb{R}^k$ for all $k$.
On the other hand, if $X$ is negatively curved and simply connected, 
$\mathscr{W}_2(X)$ does not contain
any isometric embedding of $\mathbb{R}^2$ \cite{Bertrand-Kloeckner}.

\subsubsection{Embedding powers}

First we describe a bi-Lipschitz embedding of $X^k$. This power set
can be endowed with several equivalent metrics,
for example
\[d_p\big(\bar x=(x_1,\ldots,x_k)\,,\, \bar y=(y_1,\ldots,y_k)\big)=
  \left(\sum_{i=1}^k d(x_i,y_i)^p\right)^{1/p}\]
and
\[d_\infty\big(\bar x, \bar y\big)=
  \max_{1\leqslant i\leqslant k} d(x_i,y_i)\]
which come out naturally in the proof; moreover $d_\infty$ is 
well-suited to the dynamical application below.

\begin{theo}\label{theo:embedding}
Let $X$ be any metric space, $p\in[1,\infty)$ and $k$ be any positive integer.
There exists a map $f:X^k\to\mathscr{W}_p(X)$ such that 
for all $\bar x,\bar y\in X^k$:
\[ \frac{1}{k(2^k-1)^{\frac1p}}\, d_p(\bar x,\bar y)\leqslant \wassd_p(f(\bar x), f(\bar y)) 
  \leqslant \left(\frac{2^{k-1}}{2^k-1}\right)^{\frac1p} d_p(\bar x,\bar y)\]
and that intertwines dynamical systems in the following sense:
given any measurable self-map $\varphi$ of $X$,
denoting by $\varphi_k$ the induced map on $X^k$ and 
by $\varphi_\#$ the induced map on measures, it holds
\[f\circ\varphi_k=\varphi_\#\circ f.\]
\end{theo}
Note that since $d_\infty\leqslant d_p\leqslant k^{\frac1p} d_\infty$
similar bounds hold with $d_\infty$; in fact the lower bound
that comes from the proof is in terms of $d_\infty$ and is slightly better:
\[ \frac{1}{k^{1-\frac1p}(2^k-1)^{\frac1p}}\, d_\infty(\bar x,\bar y)
\leqslant \wassd_p(f(\bar x), f(\bar y)).\]
This result is proved in Section \ref{sec:proof}.

We shall see in Section \ref{sec:constants}
that the constants cannot be improved much for general spaces, but
that for some specific spaces, a bi-Lipschitz map with a lower bound
polynomial in $k$ can be constructed. This map however does not enjoy
the intertwining property.

The explicit constants in Theorem \ref{theo:embedding} can be used to 
get information on largeness in the Minkowski sense only,
since critical parameters are designed not to grow under
countable unions. Let us give a more dynamical application
that uses the intertwining property in a crucial way. 
\begin{coro}\label{coro:dynamics}
If $X$ is compact and $\varphi:X\to X$ is a continuous map with
positive topological entropy,
then $\varphi_\#$ has positive metric mean dimension. More precisely
\[\mdim(\varphi_\#,\wassd_p)\geqslant p\frac{\entropy(\varphi)}{\log 2}.\]
\end{coro}
Metric mean dimension is a metric invariant of dynamical systems
that refines entropy for infinite-entropy ones, introduced by
Lindenstrauss and Weiss \cite{Lindenstrauss-Weiss} in link with mean dimension,
a topological invariant. The definition of $\mdim$ is recalled in Section
\ref{sec:dynamics}.

Note that the constant in Corollary \ref{coro:dynamics} is not optimal in
the case of multiplicative maps $\times d$ acting on the circle: in
\cite{Kloeckner2} we prove the lower bound $p(d-1)$ (instead of
$p\log_2 d$ here).

It is a natural question to ask whether the (topological) mean dimension of
$\varphi_\#$ is positive
as soon as $\varphi$ has positive entropy. To determine this at least 
for some map $\varphi$ would be interesting.

\subsubsection{Embedding Hilbert cubes}

Since embedding powers cannot be enough to estimate critical parameters,
we shall embed Hilbert cubes in Wasserstein spaces.
From now on, we restrict to quadratic cost (similar results
probably hold for other exponents, up to replacing 
Hilbert cubes by $\ell^p$ analogues).

\begin{theo}\label{theo:HC1}
Given any $\lambda\in(0,1/3)$ and any compact metric space
$X$, there is a continuous map $g : \hcube(X,(\lambda^n))\to \mathscr{W}_2(X)$
that is sub-Lipschitz: for some $C>0$,
\[\wassd_2(g(\bar x),g(\bar y)) \geqslant \frac{d(\bar x, \bar y)}{C}.\]
\end{theo}
The embedding we construct here is not bi-Lipschitz, but this does not matter
to get lower bounds on critical parameters.

For rectifiable enough spaces, we can use the self-similarity of the Euclidean space
to get a much stronger statement.
\begin{theo}\label{theo:HC2}
Let $X$ be any Polish metric space that admits a bi-Lipschitz embedding of a 
Euclidean cube
$I^d$ (\textit{e.g.} any manifold of dimension $d$),
and let $(a_n)$ be any $\ell^{\frac{2d}{d+2}}$ sequence of positive numbers. 
Then there is a bi-Lipschitz embedding of $\hcube(I^d,(a_n))$ into
$\mathscr{W}_2(X)$.
\end{theo}

The embedding theorems \ref{theo:HC1} and \ref{theo:HC2} have
consequences in terms of critical parameters (defined precisely in part 
\ref{part:critical}).

\begin{prop}\label{prop:largeness1}
If $X$ is any compact metric space of positive Hausdorff dimension, then
$\mathscr{W}_2(X)$ has at least intermediate size, and more precisely
\[\crit_{\mathscr{I}}\mathscr{W}_2(X)\geqslant 2,\quad
  \crit_{\mathscr{I}_2}\mathscr{W}_2(X)\geqslant \frac{\dim X}{2\log\frac{1}{3}}.\]
\end{prop}
This estimate is very far from being sharp for many spaces,
 but it has the advantage to be completely general.

The second embedding result gives a much more precise statement
when $X$ is sufficiently regular.
\begin{theo}\label{theo:manifolds}
If $X$ is a compact $d$-dimensional manifold (or any compact space having upper-Minkowski dimension
$d$ and admitting a bi-Lipschitz embedding of $I^d$),
then $\mathscr{W}_2(X)$ has power-exponential size, and more precisely
\[\crit_{\mathscr{P}}\mathscr{W}_2(X) = d.\]
\end{theo}
The upper and lower bound are proved independently under partial hypotheses,
see Propositions
\ref{prop:largeness2} and \ref{prop:upper}. A direct consequence
of Theorem \ref{theo:manifolds} is that if $X,X'$ are $d,d'$-dimensional 
manifolds with $d>d'$, then there exists no bi-Lipschitz embedding
from $\mathscr{W}_2(X)$ to $\mathscr{W}_2(X')$.

A surprise about the proof is that the methods for the upper
and the lower bound are very different and can both seem quite rough 
(see the proofs in section
\ref{sec:largeness}), but they nevertheless give the same order of magnitude.
The fact that the power-exponential critical parameter of the Wasserstein
space coincide with the dimension of the original space in the case of manifolds
is an indication that the power-exponential scale is relevant.

It is an open problem to find a relevant ``uniform'' probability measure on
$\mathscr{W}_2(X)$ (see \cite{vonRenesse-Sturm}).
Knowing the critical parameter of a space, the Carath\'eodory construction
provides a Hausdorff-like measure, which unfortunately need not be 
finite positive. One could hope to find a function such that the
Carath\'eodory construction leads to a finite positive measure, which
would then be a natural candidate to uniformity, in particular because
the construction depends only on the geometry of the space.
Our result, while far from answering the question, at least gives an idea
of the infinitesimal behavior of any such candidate: the desired function
should be very roughly of the order of magnitude of $r\mapsto \exp(-(1/r)^d)$
when $X$ is a $d$-manifold. However, it is unlikely that the Carath\'eodory
construction can be used to produce such a measure. In the quite similar case
of the space of closed subset of the interval, endowed with the Hausdorff metric,
it has indeed been proved \cite{Boardman} that no function yields a Hausdorff
like measure that is both positive and $\sigma$-finite. It would be interesting
to determine whether this result holds in the case of Wasserstein spaces.

\subsection{Largeness of closed subset spaces}

The same methods used on Wasserstein spaces can also be used to
study the space of closed subsets of a compact metric space.
We shall end the paper in section \ref{sec:Hausdorff} with the proof
of the following result.

\begin{theo}\label{theo:Hausdorff}
Let $X$ be a compact $d$-manifold (or any compact space having upper-Minkowski dimension
$d$ and admitting a bi-Lipschitz embedding of $I^d$).
Then the space $\mathscr{C}(X)$ of closed subsets of $X$, endowed with the
Hausdorff metric, has power-exponential size and more precisely
\[\crit_{\mathscr{P}}\mathscr{C}(X) = d.\]
\end{theo}

This result should be compared with those of Boardman \cite{Boardman}
and Goodey \cite{Goodey2}, which together give a refinement of Theorem 
\ref{theo:Hausdorff} when $X=[0,1]$, and of McClure \cite{McClure}
which applies to self-similar subsets of Euclidean space
that satisfy a strong separation property.

\subsubsection*{Acknowledgements} I warmly thank Antoine Gournay for
a very interesting discussion and for
introducing me to metric mean dimension, Greg Kuperberg who
suggested me that Hausdorff dimension could be generalized,
and an anonymous referee for his numerous comments that 
greatly improved the paper.

\part{A generalization of Hausdorff dimension: 
critical parameters}\label{part:critical}

\section{Carath\'eodory's construction and scales}

In this section we consider metric spaces $X$, $Y$
(assumed to be Polish, that is complete and separable, to
avoid any measurability issue) and we use the letters
$A,B$ to denote subsets of $X$.

\subsection{Carath\'eodory's construction of measures}

The starting point of our invariant is a classical construction due to
Carath\'eodory (see \cite{Mattila} for references and proofs)
that we quickly review. The idea is to
count the number of elements in coverings of $A$ by small sets $E_i$,
weighting each set by a function of its diameter.

Let $f:[0,T)\to [0,+\infty)$ be a continuous non-decreasing function
such that $f(0)=0$. Given a subset $A$ of $X$, one defines a Borel
measure by
\[\Lambda_f(A)=\lim_{\delta\to 0} \ \inf\left\{\sum_{i=1}^\infty f(\diam E_i) \,\middle|\,
A\subset \cup E_i,\ \diam E_i\leqslant\delta,\ E_i \mbox{ closed}\right\} \]
where the limit exists since the infimum is monotone. If $f(x)=x^s$, 
$\Lambda_f$ is the $s$-dimensional Hausdorff measure (up to normalization).

We shall say that $(E_i)$ is a closed covering of $A$ if it is a covering
by closed elements, and a $\delta$-covering if all $E_i$ have diameter at most
$\delta$.

\subsection{Scales and critical parameters}

We shall perform Cara\-th\'eo\-dory's construction for a family of functions,
and we need some conditions to ensure that a sharp phase transition occurs.
\begin{defi}
A \emph{scale} is a family $\mathscr{F}$ of continuous
non-decreasing functions
$f_s:[0,T_s)\to[0,+\infty)$ such that $f_s(0)=0$, where the parameter $s$ runs over
an interval $I\subset\mathbb{R}$, and which satisfies the following \emph{separation
property}:
\[\forall t>s\in I,\ \forall C\geqslant 1,\quad f_t(C r)=o_{r\to0}(f_s(r)).\]
\end{defi}

The following families are the scales we shall use below.
The \emph{polynomial scale} (or dimensional scale) is
\[\mathscr{D}:=\left(r\mapsto r^s\right)_{s\in (0,+\infty)}\]
and its critical parameter (to be defined below) is Hausdorff dimension.
The \emph{intermediate scales} (or power-log-exponential scales)
are divided into a \emph{coarse} scale
\[\mathscr{I}:=\left(r\mapsto 
  e^{-\left(\log \frac1r\right)^s}\right)_{s\in [1,+\infty)}\]
and, for each $\sigma\in [1,+\infty)$ a \emph{fine} scale
\[\mathscr{I}_{\sigma}:=\left(r\mapsto 
  e^{-s\left(\log \frac1r\right)^\sigma}\right)_{s\in (0,+\infty)}\]
note that $\mathscr{I}_1=\mathscr{D}$.
The \emph{power-exponential scale} is
\[\mathscr{P}:=\left(r\mapsto e^{-\left(\frac1r\right)^s}\right)_{s\in(0,+\infty)}\]
The parameter $s=1$ corresponds to exponential size; while one
could consider giving a more precise scale in this case, 
the family $(r\mapsto \exp(-s/r))_s$ does
not define one: it does not satisfy the separation
property, and would lead to a 
critical parameter that is not bi-Lipschitz invariant.

Consider a scale $\mathscr{F}=(f_s)_{s\in I}$ and a subset $A$ of $X$.
We have, like in the case of Hausdorff measures and with the same proof
(using the separation property only with $C=1$):
\begin{lemm}
For all parameters $t>s\in I$, if $\Lambda_{f_t}(A)>0$ then
$\Lambda_{f_s}(A)=+\infty$.
\end{lemm}
This leads to the equalities in the following.
\begin{defi}
The \emph{critical parameter} of $A$ with respect to the scale $\mathscr{F}$
is the number
\begin{eqnarray*}
\crit_\mathscr{F} A &:=& \sup\{s\in I | \Lambda_{f_s}(A)=+\infty\} \\
  &=& \sup\{s\in I \,|\, \Lambda_{f_s}(A)>0\} \\
  &=& \inf\{s\in I \,|\, \Lambda_{f_s}(A)=0\} \\
  &=& \inf\{s\in I \,|\, \Lambda_{f_s}(A)<+\infty\} 
\end{eqnarray*}
\end{defi}
Note that the critical parameter belongs to \emph{the closure} of $I$ in $\bar{\mathbb{R}}$.

\subsection{Basic properties of the critical parameter}

The critical parameter defined by any scale $\mathscr{F}$ shares many properties
with the Hausdorff dimension.

\begin{prop}
The following properties hold:
\begin{itemize}
\item (monotonicity) if $A\subset B\subset X$, then $\crit_\mathscr{F} A\leqslant \crit_\mathscr{F} B$,
\item (countable union) for any countable family of sets $A_i\subset X$,
      \[\crit_\mathscr{F}(\cup A_i)=\sup_i \crit_\mathscr{F} A_i,\]
\item (Lipschitz monotonicity) If there is a sub-Lipschitz map from $X$
      to another metric space $Y$, then 
      \[\crit_\mathscr{F} X \leqslant \crit_\mathscr{F} Y.\]
\item (Lipschitz invariance) if there is a bi-Lipschitz map from $X$
      onto another metric space $Y$, then
      \[\crit_\mathscr{F} X =\crit_\mathscr{F} Y.\]
\end{itemize}
\end{prop}

\begin{proof}
The monotonicity and countable union properties are straigthforward 
since $\Lambda_{f_s}$ is a measure for all $s$.
The Lipschitz monotonicity and Lipschitz invariance are proved
just like the invariance of Hausdorff dimension, using
the separation property.

More precisely, let $g:X\to Y$ be a sub-Lipschitz map: for some
$D>0$ and all $x,x'\in X$,
\[d(g(x),g(x')) \geqslant D d(x,x')\]

Given any countable closed $D\delta$-covering $(F_i)$ of $Y$,
the sets $E_i=\overline{g^{-1}(F_i)}$
are closed, of diameter at most
$D^{-1}\diam F_i\leqslant \delta$ and cover $X$.
By the separation property, given any $s<t$ in the parameter set of $\mathscr{F}$,
there is a $\delta_0$ such that
for all $r\in(0,\delta_0)$ we have $f_t(D^{-1}r)\leqslant f_s(r)$.
If $\Lambda_{f_s}(Y)=0$, then we can find coverings $(F_i)$ of $Y$
of arbitrarily low diameter making $\sum f_s(\diam F_i)$
arbitrarily low.
It follows that the corresponding coverings $(E_i)$
of $X$ make $\sum f_t(\diam E_i)$ arbitrarily low, so
that $\Lambda_{f_t}(X)=0$. Letting $s$ and $t$ approach
the critical parameter of $Y$ shows that  
\[\crit_\mathscr{F} Y \geqslant \crit_\mathscr{F} X\]

If there is a bi-Lipschitz equivalence between $X$ and $Y$,
we get the other inequality by symmetry.
\end{proof}

\section{Estimations tools}

Let us give two tools to estimate the critical parameter of a given
set. Both are direct analogues of standard tools used for Hausdorff
dimension. We consider here a fixed Polish metric space $X$ and
a given scale $\mathscr{F}=(f_s)_{s\in I}$.

\subsection{Upper bounds via growth of coverings}

The most evident way to measure the size of a compact set $A$
is to consider the growth of the minimal number $N(A,\varepsilon)$ 
of radius $\varepsilon$ balls needed to cover $A$ when $\varepsilon \to 0$.
If $N(A,\varepsilon)$ is roughly $(1/\varepsilon)^d$, more precisely
if
\[\lim_{\varepsilon\to0} \frac{\log N(A,\varepsilon)}{\log (1/\varepsilon)} = d\]
then one says that $X$ has Minkowski dimension (or M-dimension for short, also
called box dimension) equal to $d$.
The limit need not exist, and one defines the upper and lower M-dimensions
by replacing it by an infimum or supremum limit. Equivalently, one can define
these dimensions by
\begin{eqnarray*}
\udim(A) &=& \inf \left\{s>0 \,\middle|\, 
  \limsup_{\varepsilon\to 0} N(A,\varepsilon)\varepsilon^s <+\infty\right\} \\
\ldim(A) &=& \inf \left\{s>0 \,\middle|\, 
  \liminf_{\varepsilon\to 0} N(A,\varepsilon)\varepsilon^s <+\infty\right\}
\end{eqnarray*}
which is much more easily generalized to arbitrary scales.

\begin{defi}
The \emph{lower} and \emph{upper Minkowski critical parameter}
of a compact set $A\subset X$ with respect to the scale $\mathscr{F}$
are defined as
\begin{eqnarray*}
\ucrit_\mathscr{F}(A) &:=& \inf \left\{s\in I \,\middle|\, 
  \limsup_{\varepsilon\to 0} N(A,\varepsilon) f_s(\varepsilon) <+\infty\right\}\\
\lcrit_\mathscr{F}(A) &:=& \inf \left\{s\in I \,\middle|\, 
  \liminf_{\varepsilon\to 0} N(A,\varepsilon) f_s(\varepsilon) <+\infty\right\}
\end{eqnarray*}
\end{defi}

It is clear from the definition that 
$\lcrit_\mathscr{F}(A)\leqslant \ucrit_\mathscr{F}(A)$,
and there are several other equivalent ways to define the Minkowski critical 
parameters, for example
\[\lcrit_\mathscr{F}(A) = \sup \left\{s\in I \,\middle|\, 
  \liminf_{\varepsilon\to 0} N(A,\varepsilon) f_s(\varepsilon) >0\right\}\]

The following result enables one to get upper bounds on the critical parameter.
\begin{prop}
The following inequality always holds: 
\[\crit_\mathscr{F}(A)\leqslant\lcrit_\mathscr{F}(A).\]
\end{prop}

\begin{proof}
For all positive $\varepsilon$, there is a covering $(B_i)$ of $A$
by $N(A,\varepsilon)$ balls of radius $\varepsilon$. Given any 
$t>s>\lcrit_\mathscr{F}(A)$ we have 
\[\sum f_t(\diam B_i)\leqslant N(A,\varepsilon) f_t(2\varepsilon)
  \leqslant N(A,\varepsilon) f_s(\varepsilon)\]
as soon as $\varepsilon$ is small enough. Passing to an infimum limit,
we get $\Lambda_{f_t}(A)=0$ and thus $\crit_\mathscr{F} A \leqslant t$.
\end{proof}

Unfortunately, there is no way to have a lower bound of the critical
parameter in terms of these Minkowski versions. The classical counter-example
is the set $\{0,1,1/2,1/3,\dots\}$ that has Minkowski dimension $1/2$
but is countable, thus has Hausdorff dimension $0$. This is one of
the reasons to introduce Hausdorff dimension and 
more general critical parameter: Minkowski critical parameters can grow significantly
under countable union. 

They however share the other properties of critical parameters.
\begin{prop}
The upper and lower Minkowski critical parameter
satisfy the monotonicity and Lipschitz invariance properties:
\begin{itemize}
\item if $A\subset B\subset X$, then 
  \[\lcrit_\mathscr{F}(A)\leqslant \lcrit_\mathscr{F}(B) \quad\mbox{and}\quad
    \ucrit_\mathscr{F}(A)\leqslant \ucrit_\mathscr{F}(B),\]
\item if there is a bi-Lipschitz equivalence $A\to B$, then
      \[\lcrit_\mathscr{F}(A)=\lcrit_\mathscr{F}(B)\quad\mbox{and}
      \quad \ucrit_\mathscr{F}(A)=\ucrit_\mathscr{F}(B).\]
\end{itemize}
\end{prop}
We do not give the easy proof of this result, but note that for the bi-Lipschitz
invariance, again one needs the full power
of the separation property for scales.

In order to compute $\lcrit$ and $\ucrit$, one can also use packings:
denoting by $P(A,\varepsilon)$ the maximal number of points in $A$ that are
pairwise at distance at least $\varepsilon$, we indeed have 
$N(A,2\varepsilon)\leqslant P(A,\varepsilon)$ and 
$P(A,2\varepsilon)\leqslant N(A,\varepsilon)$. Here, once again, the 
strong separation property is vital to ensure that the factor $2$ is harmless.

\subsection{Lower bounds via Frostman's lemma}

Finding a large packing of balls in $A$
is not sufficient to bound the critical parameter from below but,
as for the Hausdorff dimension, a close analogue that is sufficient
is to exhibit a measure with small growth.

\begin{prop}[Frostman's Lemma]
For all Borel subset $A$ of $X$, if there is a Borel probability measure $\mu$ 
concentrated on $A$ and a positive constant $C$ such that for all 
$x\in A$ and all $r>0$
\[\mu(B(x,r))\leqslant C f_s(r)\]
then $\Lambda_{f_s}(A)>0$ (and in particular
$\crit_\mathscr{F}(A)\geqslant s$).
Moreover the converse holds.
\end{prop}

The proof can be found for example in \cite{Mattila}.
The difficult part is the converse, while
the very useful direct part is straightforward.

\section{Critical parameters of Hilbert cubes}

Let us now use the previous tools to compute critical parameters
for the Hilbert cubes defined in the introduction. Here $X$
is assumed to be compact.

The topology of a Hilbert cube $\hcube(X;\bar a)$
is the product topology, in particular it is compact.
It need not be infinite dimensional
in general; for example if $X$ is finite and $\bar a$ is geometric,
then $\hcube(X,\bar a)$ is a finite-dimensional, self-similar Cantor set.

We shall estimate critical parameters for two different kind
of coefficients $\bar a$; in both cases the upper bound is obtained
with the same method, so let us give a technical
lemma to avoid repetition.
\begin{lemm}\label{lemm:techabove}
Let $(X,d)$ be a compact metric space of finite, positive
upper Minkowski dimension $s$ and
let $\bar a=(a_n)_{n\geqslant 1}$
be an $\ell^2$ sequence of positive numbers. If $L:(0,1)\to\mathbb{N}^*$
is a non-increasing function such that
\[\sum_{n > L(\varepsilon)} a_n^2\leqslant \frac{\varepsilon^2}{(\diam X)^2 }\]
then for all $\eta>0$ and all $\varepsilon$ small enough compared
to $\eta$, we have
\begin{equation}
\log N(\hcube(X,\bar a),\varepsilon)\leqslant
  (s+\eta)\left(\log\prod_{n=1}^{L(\varepsilon/2)} a_n+\frac12
  \log L(\frac\varepsilon2)!+L(\frac\varepsilon2)\log\frac1\varepsilon
  \right).
\label{eq:techabove}
\end{equation}
\end{lemm}

\begin{proof}
Let $s'=s+\eta$. By definition of upper M-dimension, there is a constant
$C$ such that for all $\varepsilon<1$,
$N(X,\varepsilon)\varepsilon^{s'}\leqslant C$.
We shall construct a covering of the Hilbert cube
from coverings at different scales of $X$.
 Denoting
by $X_{a_n}$ the space $X$ endowed with the metric $a_n d$, we
have $N(X_{a_n},\varepsilon)=N(X,\varepsilon/a_n)$, thus for all
$\varepsilon$ smaller than $\max a_n$ and each $n$, we can find
a family of $C(2Ca_n  \sqrt{n}\log n/\varepsilon)^{s'}$
points $(x_n^i)_i$ such that
every $x\in X_{a_n}$ is at distance at most $\varepsilon/(2C\sqrt{n}\log n)$
from one of them. The use of the sequence $(\sqrt{n}\log n)$ will
become clear in a moment; what is important is that it increases not
too fast, but its inverse is $\ell^2$.

Now any point $(x_1,x_2,\dots)$ in $\hcube(X,\bar a)$
is at distance at most $\varepsilon/2$ from
$(x_1,\dots,x_{L(\varepsilon/2)},0,0,\dots)$, which is itself at
distance at most 
\[\frac{\varepsilon}{2C}\left(\sum \frac{1}{n \log^2 n}\right)^{1/2}\leqslant
  \frac\varepsilon2\]
(up to enlarging $C$ if needed) from one of the points
$(x_1^{i_1},\dots,x_{L(\varepsilon/2)}^{i_L(\varepsilon/2)},0,\dots)$.
We get
\[N(\hcube(X,\bar a),\varepsilon)\leqslant \prod_{n=1}^{L(\varepsilon/2)}
  C\left(\frac{2C a_n \sqrt{n}\log n}{\varepsilon}\right)^{s'}\]
and we only have left to take the logarithm; two terms can be removed
up to doubling $\eta$: one
proportional to $L(\varepsilon/2)$, absorbed by the 
$L(\varepsilon/2)\log 1/\varepsilon$ term, and one
proportional to $\sum_1^{L(\varepsilon/2)} \log\log n$, absorbed
by the $\log (L(\varepsilon/2)!)$ term. Note that this last comparison
is of course very inefficient, but it avoids adding a 
$L(\varepsilon/2)\log\log L(\varepsilon/2)$ term
to the formula and the $\log (L(\varepsilon/2)!)$ term must be present
anyway due to the presence of $\sqrt n$ in the product above.
\end{proof}

When $\bar a$ decays exponentially,
the Hilbert cube has intermediate size and its fine critical parameter can
be determined.
\begin{prop}\label{prop:expdecay}
Let $X$ be any compact metric space and let $\lambda\in(0,1)$.
We have
\[\frac{\dim X}{2\log\frac1\lambda}\leqslant 
  \crit_{\mathscr{I}_2}\hcube(X,(\lambda^n))\leqslant 
  \ucrit_{\mathscr{I}_2}\hcube(X,(\lambda^n)) \leqslant
  \frac{\udim X}{2\log\frac1\lambda}\]
In particular, if $X$ has positive and finite Hausdorff and upper
Minkowski dimension, then
\[\crit_{\mathscr{I}}\hcube(X,(\lambda^n))=2\]
\end{prop}
In particular, when $0<\udim X=\dim X<+\infty$
the $2$-fine intermediate
critical parameter of the Hilbert cube is equal to $\dim X/(2\log 1/\lambda)$.

\begin{proof}
We denote by $H$ the generalized Hilbert cube under study.
Note that both inequalities are trivial when $\dim X=0$ and, respectively,
$\udim X=+\infty$. We therefore assume otherwise.

Using the notation Lemma \ref{lemm:techabove}, one can choose $L$ such that
\[L(\varepsilon/2)\sim \frac{\log\frac1\varepsilon}{\log\frac1\lambda}\]
where $f\sim g$ means asymptotic equivalence: $f=g + o(g)$. 
Then in \eqref{eq:techabove} the second term is negligible
(of the order of $\log 1/\varepsilon \log\log 1/\varepsilon$)
compared to the first and third ones, and for all $s>\udim X$ we get
when $\varepsilon$ is small enough (up to invoking the lemma for a slightly 
smaller $s$):
\[\log N(H,\varepsilon)\leqslant s\left(
  -\frac{\left(\log\frac1\varepsilon\right)^2}{2\log\frac1\lambda} +
  \frac{\left(\log\frac1\varepsilon\right)^2}{\log\frac1\lambda}\right)
\]
so that $\ucrit_\mathscr{I} H\leqslant 2$ and 
$\ucrit_{\mathscr{I}_2} H\leqslant \udim X/(2\log 1/\lambda)$.

For all $0<t<\dim X$, there is a Borel probability measure $\nu$ on $X$
such that $\nu(B(x,r))\leqslant C r^{t}$ for all $r$. Such a measure
exists by Frostman's lemma since the $t$-dimensional Hausdorff measure
of $X$ is infinite, hence positive. Now $\mu:=\otimes_{n=1}^{+\infty} \nu$
is a Borel probability measure on $\hcube(X,(\lambda^n))\simeq X^{\mathbb{N}}$,
and for all $r>0$, all functions $M:\mathbb{R}^+\to\mathbb{N}$ and all $\bar x$ we have
\[B(\bar x,r)\subset \prod_{n=1}^{M(r)} B_{\lambda^n}(x_n,r)\times
   X\times X\times\dots\]
where $B_{\lambda^n}(x_n,r)$ is the ball in the scaled space
$X_{\lambda^n}$, and is therefore equal as a set to $B(x_n,r\lambda^{-n})$.
This ball has $\nu$-mesure at most $C(r\lambda^{-n})^{t}$ so that we get
\[\log\mu(B(\bar x,r))\leqslant t\left(-M(r)\log\frac1r
  +\frac{M(r)^2}2\log \frac1\lambda\right)+O(M(r))\]
The optimal choice is then to take 
\[M(r)\sim\frac{\log\frac1r}{\log\frac1\lambda}\]
so that 
\[\log\mu(B(\bar x,r))\leqslant -\frac{t}{2\log\frac1\lambda}\left(\log\frac1r\right)^2
+O\left(\log\frac1r\right)\]
Using Frostman's lemma and letting $t$ go to $\dim X$ we get
\[\crit_\mathscr{I_2} H\geqslant \dim X/(2\log 1/\lambda)\]
and in particular $\crit_\mathscr{I} H\geqslant 2$. 
\end{proof}

When $\bar a$ decays polynomially,
the corresponding Hilbert cube over any space of positive and finite
dimension has power-exponential size, mostly independant of the geometry of
$X$. Note that we shall need more precision than before when using Frostman's
lemma.
\begin{prop}\label{prop:polydecay}
Let $X$ be any compact metric space and let $\alpha>1/2$.
If $X$ has positive Hausdorff dimension, then
\[\frac{2}{2\alpha-1}\leqslant \crit_{\mathscr{P}}\hcube(X,(n^{-\alpha}))\]
and if $X$ has finite upper Minkowski dimension then
\[\ucrit_{\mathscr{P}}\hcube(X,(n^{-\alpha})) \leqslant\frac{2}{2\alpha-1}.\]
\end{prop}
In particular, when $X$ has positive and finite Hausdorff and upper
Min\-kow\-ski dimensions, the power-exponential critical parameter
of $\hcube(X,(n^{-\alpha}))$ is equal to $2/(2\alpha-1)$.

\begin{proof}
Using the notation of Lemma \ref{lemm:techabove}, $L$ can be chosen
such that there are constants $C<D$ satisfying
\[C\left(\frac1\varepsilon\right)^{\frac{2}{2\alpha-1}}\leqslant
L(\varepsilon)\leqslant D\left(\frac1\varepsilon\right)^{\frac{2}{2\alpha-1}}.\]
For all $s$ greater than the upper M-dimension of
$X$ and all small enough $\varepsilon$ we have
(recalling that according to Stirling's formula, $\log m!=m\log m+O(m)$)
\[\log N(H,\varepsilon)\leqslant 
  s(D-C)\left(\frac2\varepsilon\right)^{\frac{2}{2\alpha-1}}
  \log\frac1\varepsilon\]
For all $t>2/(2\alpha-1)$, the quantity
$N(H,\varepsilon)\exp(-(1/\varepsilon)^t)$ is therefore bounded. It follows that
\[\ucrit_{\mathscr{P}}\hcube(X,(n^{-\alpha}))\leqslant\frac{2}{2\alpha-1}.\]

To get the lower bound, we start by assuming $\dim X>1$ (otherwise,
take $p>1/\dim X$ so that $\dim X^p >1$ and observe that there is
a bi-Lipschitz embedding from $\hcube(X^p,(n^{-\alpha}))$ to
$\hcube(X,(n^{-\alpha}))$).

From Frostman's lemma there is a non-zero Borel probability measure $\nu$ on $X$
such that $\nu(B(x,r))\leqslant C r$ for all $r$. As before we define 
$\mu:=\otimes_{n=1}^{+\infty} \nu$
which is a Borel probability measure on $H=\hcube(X,(n^{-\alpha}))\simeq X^{\mathbb{N}}$.
We want to precisely estimate the $\mu$-measure of small balls in $H$. Fix a point
$\bar x\in H$. For convenience, we introduce the notation $\bar a=(n^{-\alpha})_n$,
$\bar a^k=(n^{-\alpha})_{n\geqslant k}$ and we define similarly $\bar x^k$. Let
also $S_{a_n}(x,r)$ be the sphere of center $x$ and radius $r$ in $X_{a_n}$.
We can write 
\[B(\bar x,r)=\bigcup_{0\leqslant r_1\leqslant r} S_{a_1}(x_1,r_1)\times
  B\big(\bar x^2,\sqrt{r^2-r_1^2}\big)\]
where the right factor is a ball of $\hcube(X,\bar a^2)$. Denoting by
$\sigma$ the push-forward of the measure $\nu$
by the map $x\mapsto d_{a_1}(x_1,x)$,
we have
\[\nu(B(x_1,r))=\int_0^r \,\sigma(dr_1)\]
and by Fubini's theorem
\[\mu(B(\bar x,r)) = \int_0^r \mu\left(B\big(\bar x^2,\sqrt{r^2-r_1^2}\big)\right)
  \,\sigma(dr_1)\]
We know that $\int_0^r \,\sigma(dr_1)\leqslant C r/a_1$ for all $r>0$,
thus there exists a coupling measure $\Pi$ on $\mathbb{R}^+\times\mathbb{R}^+$
supported on $\{(u,v)| u\geqslant v\}$ such that its first marginal is equal
to $\sigma$ and its second marginal is lesser than or equal to $(C/a_1)dv$
(with $dv$ the Lebesgue measure). One indeed
can take for $\Pi$ the increasing rearrangement between these two measures
(see e.g. \cite{Villani2} page 7 for a definition).
Using that the left factor in the following integrand is non-increasing,
we get
\begin{eqnarray*}
\mu(B(\bar x,r)) &=& 
  \int_0^r \mu\left(B\big(\bar x^2,\sqrt{r^2-r_1^2}\big)\right) \,\sigma(dr_1)\\
  &=& \int_0^r\int_{\mathbb{R}^+} \mu\left(B\big(\bar x^2,\sqrt{r^2-r_1^2}\big)\right)
      \,\Pi(dr_1 dv) \\
  &\leqslant&  \int_0^r\int_{\mathbb{R}^+} 
      \mu\left(B\big(\bar x^2,\sqrt{r^2-v^2}\big)\right)
      \,\Pi(dr_1 dv) \\
  &\leqslant& \int_0^r \mu\left(B(\bar x^2,\sqrt{r^2-v^2})\right) (C/a_1)dv
\end{eqnarray*}
Assume, given an integer $M$, that there is a constant $C_M^{\bar a}$
such that $\mu(B(\bar x,r))\leqslant C_M^{\bar a} r^M$ for all $r$. Then,
using a change of variable $v=r\cos \theta$, the above inequality
yields
\[\mu(B(\bar x,r))\leqslant C_M^{\bar a^2}\frac{C}{a_1} 
 \left(\int_0^{\frac\pi2} sin^{M+1}\theta \,d\theta \right) r^{M+1}\]
We know that the Wallis integral is asymptotically equivalent
to $\sqrt{\pi/2(M+1)}$, so that there is a positive constant $D$
depending on $C$ such that
\[\mu(B(\bar x, r))\leqslant \frac{D^M}{\sqrt{M!}\prod_1^M a_n} r^M\]

Defining an integer valued function $M$ such that $M(r)\sim r^{-\beta}$,
we get
\[\log\mu(B(\bar x, r))\leqslant \left(\beta(\alpha-\frac12)-1\right) 
  M(r)\log\frac1r+ O(M(r))\]
so that whenever $\beta<2/(2\alpha-1)$, we have
\[\log\mu(B(\bar x, r))\leqslant -E\left(\frac1r\right)^\beta\log\frac1r\]
for some positive constant $E$, and we deduce from Frostman's lemma that
$\crit_\mathscr{P} H \geqslant 2/(2\alpha-1)$.
\end{proof}

Corollary \ref{coro:application} from the introduction
is a direct consequence of the above result.
\begin{proof}[Proof of Corollary \ref{coro:application}]
Assume there is a bi-Lipschitz embedding from the Hilbert cube
$\hcube(X,(n^{-\alpha}))$ to $\hcube(Y,(n^{-\beta}))$
where $X$ has positive Hausdorff dimension and $Y$ has finite
upper Minkowski dimension. Then by Proposition \ref{prop:polydecay}
and the monotonicity property, we have
\[\frac2{2\alpha-1}\leqslant \crit_\mathscr{P} \hcube(X,(n^{-\alpha}))
\leqslant \ucrit_\mathscr{P} \hcube(Y,(n^{-\beta}))\leqslant \frac2{2\beta-1}\]
which implies $\beta\leqslant \alpha$.
\end{proof}

\part{Largeness of Wasserstein spaces}

\section{Wasserstein spaces}\label{sec:recalls}

For a detailed introduction on optimal transport, the interested reader can for
example consult \cite{Villani}, or \cite{Santambrogio} for a more concise 
overview. 
Optimal transport is about moving a given amount of material from one
distribution to another with the least total cost, where the cost
to move a unit of mass between two points is given by a cost function.
Here the cost function is related to a metric, and optimal transport
gives a metric on a space of measures. Let us give a few precise definitions and 
the properties we shall need.

Given an exponent $p\in[1,\infty)$, if $(X,d)$ is a general metric space,
always assumed to be Polish (complete separable),
and endowed with its Borel 
$\sigma$-algebra, its $L^p$ \emph{Wasserstein space}  is
the set $\mathscr{W}_p(X)$ of (Borel) probability measures $\mu$ on $X$ whose $p$-th 
moment is finite:
\[\int d(x_0,x)^p \,\mu(dx)<\infty\qquad\mbox{ for some, hence all }x_0\in X\]
endowed with the following metric: given $\mu,\nu\in\mathscr{W}_p(X)$ one sets
\[\wassd_p(\mu,\nu)=\left(\inf_\Pi \int_{X\times X} d(x,y)^p\, 
  \Pi(dx dy)\right)^{1/p}\]
where the infimum is over all probability measures $\Pi$ on $X\times X$
that project to $\mu$ on the first factor and to $\nu$ on the second one.
Such a measure is called a transport plan between $\mu$ and $\nu$, and is
said to be optimal when it achieves the infimum. 
The function $d^p$ is called the cost function, and the value of
$\int_{X\times X} d(x,y)^p\,\Pi(dx dy)$ is the \emph{total cost} of $\Pi$.

In this setting, an optimal
transport plan always exists. Note that when $X$ is compact, the set $\mathscr{W}_p(X)$
is equal to the set $\mathscr{P}(X)$ of all probability measures on $X$
and $\wassd_p$ metrizes the weak topology.

The name ``transport plan'' is suggestive: it is a way to describe what amount of
mass is transported from one region to another.

One very useful tool to study optimal transport is \emph{cyclical mononotonicity}.
Given a cost function $c$ ($=d^p$ here) on $X\times X$, one says
that a set $S\subset X\times X$ is ($c$-)cyclically monotone if
for all families of pairs $(x_0,y_0),\dots,(x_k,y_k)\in S$, one
has
\[c(x_0,y_0)+\dots+c(x_k,y_k)\leqslant c(x_0,y_1)+\dots+c(x_{k-1},y_k)+c(x_k,y_0)\]
in words, one cannot reduce the total cost to move a unit amount of mass
from the $x_i$ to the $y_i$ by permuting the target points. A transport
plan $\Pi$ is said to be cyclically monotone if its support is. Using
continuity of the cost we use here, it is easy to see that an optimal
transport plan must be cyclically monotone. It is a non-trivial result 
that the reciprocal is also true, see \cite{Villani}.

\section{Embedding powers}\label{sec:proof}

This section is logically independent of the rest of the article.
We prove Theorem \ref{theo:embedding} and consider its optimality
and its dynamical consequence.

\subsection{Proof of Theorem \ref{theo:embedding}}

The first power of $X$ embeds isometrically by $x\to\delta_x$ where $\delta_x$ 
is  the Dirac mass at a point. To construct an embedding $f$ of
a higher power of $X$ into its Wasserstein space,  the idea
is to encode a tuple by a measure supported
on its elements, without adding any extra symmetry:
one should be able to distinguish $f(a,b,\ldots)$ from $f(b,a,\ldots)$. 
Define the map 
\begin{eqnarray*}
f:\qquad\qquad\qquad X^k &\to& \mathscr{W}_p(X) \\
\bar x = (x_1,\ldots,x_k) &\mapsto& \alpha \sum_{i=1}^k \frac1{2^i} \delta_{x_i}
\end{eqnarray*}
where $\alpha=1/(1-2^{-k})$ is a normalizing constant. This choice
of masses moreover ensures that different subsets of the tuple have different masses.
This map obviously has the intertwining property since
 $\varphi_\#(\delta_x)=\delta_{\varphi(x)}$.

\begin{lemm}
The map $f$ is $(\alpha/2)^{\frac1p}$-Lipschitz when
$X^k$ is endowed with the metric $d_p$.
\end{lemm}

\begin{proof}
There is an obvious transport plan from an image $f(\bar x)$ to another $f(\bar y)$, given by
$\alpha\sum_i 2^{-i} \delta_{x_i}\otimes\delta_{y_i}$. Its $L^p$ cost is 
\[\alpha\sum_i 2^{-i} d(x_i,y_i)^p \leqslant \alpha/2 \sum_i d(x_i,y_i)^p\]
so that $\wassd_p(f(\bar x),f(\bar y)) \leqslant (\alpha/2)^{\frac1p} d_p(\bar x,\bar y)$.
\end{proof}

Our goal is now to bound $\wassd_p(f(\bar x),f(\bar y))$ from below. 
The very formulation of the Wasserstein metric makes it more difficult
to give lower bounds than upper bounds. One classic way around this issue
is to use a dual formulation (Kantorovich duality) that expresses
the minimal cost in terms of a supremum. Here we give a more direct,
combinatorial approach based on cyclical monotonicity.

The cost of all
transport plans below are computed with respect to the cost $d^p$, where $p$ is fixed.

\subsubsection{Labelled graphs}

To describe transport plans, we shall use \emph{labelled graphs}, defined as tuples
$G=(V,E,m,m_0,m_1)$ where $V$ is a finite subset of $X$, $E$ is a set of
pairs $(x,y)\in V^2$ where $x\neq y$ (so that $G$ is an oriented graph without loops),
$m$ is a function $E\to[0,1]$ and $m_0, m_1$ are functions $V\to[0,1]$. 
An element of $V$ will usually be denoted by $x$ if it is thought of as a starting point,
$y$ if it is thought of as a final point, and $v$ if no such assumption is made.

To any transport
plan between finitely supported measures, one can associate a labelled graph
as follows.
\begin{defi}
Let $\mu,\nu$ be probability measures supported on finite sets $A,B\subset X$
and let $\Pi$ be any transport plan from $\mu$ to $\nu$. We define a labelled
graph $G^\Pi$ by: $V^\Pi=A\cup B$, 
\[E^\Pi=\supp\Pi\setminus\Delta=\big\{(x,y)\in X^2\,\big|\,
  x\neq y\mbox{ and }\Pi(\{x,y\})>0\big\},\] 
$m^\Pi(x,y)=\Pi(\{x,y\})$, $m^\Pi_0(x)=\mu(\{x\})$ and $m^\Pi_1(y)=\nu(\{y\})$.
\end{defi}
In other words, the graph encodes the initial and final measures and the amount
of mass moved from any given point in $\supp\mu$ to any given point in $\supp\nu$.
The transport plan itself can be retrieved from its graph; for example its cost
is 
\[c_p(\Pi)=\sum_{e\in E} m^\Pi(e)d(e^-,e^+)^p\]
where $e^-$ and $e^+$ are the starting and ending points of the edge $e$.

Not every labelled graph encodes a transport plan between two measures. We say that
$G$ is \emph{admissible} if:
\begin{itemize}
\item $\sum_V m_0(v)=\sum_V m_1(v)=1$, 
\item for all $e\in E$, $m(e)>0$,
\item for all $v\in V$, $m_0(v)+\sum_{e=(x,v)\in E} m(e)-\sum_{e=(v,y)\in E} m(e)=m_1(v)$
  (this is mass invariance),
  $\sum_{e=(x,v)\in E} m(e)\leqslant m_1(v)$ and $\sum_{e=(v,y)\in E} m(e)\leqslant m_0(v)$.
\end{itemize}
A labelled graph is admissible if and only if it is the graph of some transport plan.
The next steps of the proof shall give some information on the graphs of optimal plans.

\subsubsection{The graph of some optimal plan is a forest}

Let us introduce some notation related to a given labelled graph
$G$. A \emph{path} is a 
tuple of edges $P=(e_1,\ldots,e_l)$ such that $e_i$ has an endpoint in common
with $e_{i+1}$ for all $i$. If moreover $e_i^+=e_{i+1}^-$ holds for all $i$,
we say that $P$ is an \emph{oriented path}.
We define the \emph{unitary cost} of $P$ as the cost of a unit mass 
travelling along $P$, that is $c(P)=\sum_{i=1}^{l} d(e_i^-,e_i^+)^p$, and
the \emph{flow} of $P$ as the amount of mass travelling along $P$,
that is $\phi(P)=\min_i m(e_i)$. Cycles and oriented cycles are defined
in an obvious, similar way; a graph is a \emph{forest} if it contains no cycle.

\begin{lemm}
If $\Pi$ is an optimal plan between any two finitely supported measures
$\mu,\nu$, then $G^\Pi$ contains no oriented cycle.
\end{lemm}

\begin{proof}
This is a direct consequence of the cyclic monotonicity of optimal plans:
if there were points $v_1,v_2,\ldots,v_n$ in $V^\Pi$ such that $v_n=v_1$ and
$m(i):=m^\Pi(v_i,v_{i+1})>0$ for all $i<n$, then by subtracting the minimal value of
$m_i$ from each of them one would get an new admissible labelled graph 
with $m_0=m_0^\Pi$ and $m_1=m_1^\Pi$ and cost less than the cost of $G^\Pi$.
This new graph would give a new transport plan from $\mu$ to $\nu$,
cheaper than $\Pi$.
\end{proof}

An optimal plan can a priori have non-oriented cycles, but up to changing the plan (without
changing its cost), we can assume it does not.
\begin{lemm}
Between any two finitely supported measures $\mu,\nu$, there is an optimal
plan $\Pi$ such that $G^\Pi$ is a forest.
\end{lemm}

\begin{proof}
Let $\Pi$ be any optimal plan from $\mu$ to $\nu$, and let $G_0=G^\Pi$ be its graph.

A non-oriented cycle is determined by two sets of vertices $x_1,\ldots,x_n$ and $y_1,\ldots, y_n$
and two sets of oriented paths $P_i : x_i\to y_i$, $Q_i : x_i \to y_{i+1}$ where
$y_{n+1}:=y_1$, see Figure \ref{fig:cycle}.

\begin{figure}\begin{center}
\input{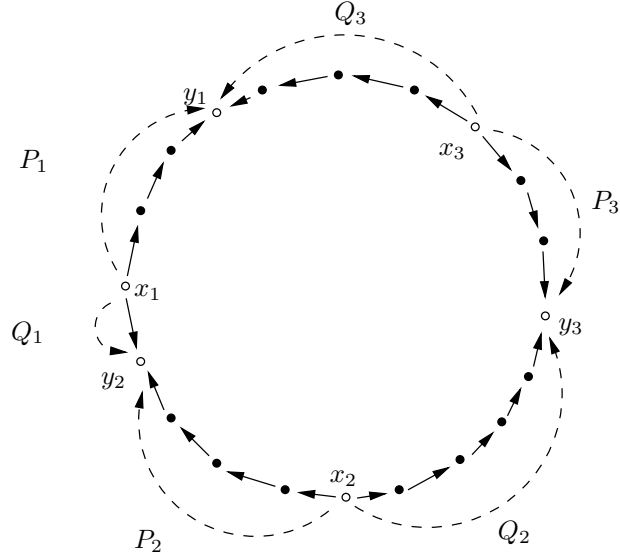}
\caption{A non-oriented cycle: $x_i$'s and $y_i$'s are the vertices where the edges
  change orientation.}\label{fig:cycle}
\end{center}\end{figure}

Consider a minimal non-oriented cycle of $G_0$,
so that no two paths among all $P_i$'s and $Q_i$'s share an edge.

One can construct a new admissible labelled graph $G_1$, with the same vertex labels
$m_0$ and $m_1$ as $G$, by adding a small
$\varepsilon$ to all $m(e)$ where $e$ appears in  some $P_i$, and subtracting the same
$\varepsilon$ from all $m(e)$ where $e$ appears in some $Q_i$. This operation adds
$\varepsilon$ to $\phi(P_i)$ and $-\varepsilon$ to $\phi(Q_i)$, thus it adds
$\varepsilon\sum_i c(P_i)-c(Q_i)$ to the cost of $\Pi$.

Since $\Pi$ is optimal, one cannot reduce its cost by this operation. This implies
that $\sum_i c(P_i)-c(Q_i) = 0$. By operating as above with $\varepsilon$ equal to
plus or minus the minimal value of all $m(e)$ where $e$ appears in a $P_i$ or in 
a $Q_i$, one designs the wanted new admissible graph $G_1$.

Now, $G_1$ has its edge set included in the edge set of $G$, with at least one
less oriented cycle. By repeating this operation, one constructs an admissible
labelled graph $G$ without cycle, that has the same total cost and the same vertex labels
as $G_0$. The transport plan defined by $G$ is therefore optimal, from
$\mu$ to $\nu$.
\end{proof}

The non-existence of cycles has an important consequence.
\begin{lemm}\label{lemm:forest}
Let $\Pi$ be a transport plan between two finitely supported measures
$\mu$ and $\nu$, whose graph is a forest. If there is some real number $r$
such that all $m^\Pi_0(v)$ and all $m^\Pi_1(v)$ are integer multiples of
$r$, then all $m^\Pi(e)$ are integer multiples of $r$.
\end{lemm}

\begin{proof}
Let $G_0=G^\Pi=(V,E,m,m_0,m_1)$. If $G_0$ has no edge, then we are done. Otherwise,
$G_0$ has a leaf, that is a vertex $x_0$ connected to exactly one
vertex $y_0$, by an edge $e_0$. Assume for example
that $e_0=(x_0,y_0)$ (the other case is treated similarly). 
Then $m(e_0)=m_0(x_0)-m_1(x_0)$ is an integer
multiple of $r$.

Define $G_1=(V, E\setminus\{e_0\},m',m'_0,m'_1)$
where:
\begin{itemize}
\item $m'(e)=m(e)$ for all $e\in E\setminus\{e_0\}$, 
\item $m'_0(x_0)=m_0(x_0)+m(e_0)$,
\item $m'_0(x)=m_0(x)$ for all $x\in V\setminus\{x_0\}$,
\item $m'_1(y_0)=m_1(y_0)-m(e_0)$,
\item $m'_1(y)=m_1(y)$ for all $y\in V\setminus\{y_0\}$.
\end{itemize}
Then $G_1$ is still admissible (with different starting and
ending measures $\mu'$ and $\nu'$, though), and all $m'_0(v), m'_1(v)$ are integer
multiples of $r$. By induction, we are reduced to the case of an
edgeless graph.
\end{proof}

\subsubsection{End of the proof}

Now we are ready to bound $\wassd_p(f(\bar x),f(\bar y))$ from below in terms 
of $d_\infty(\bar x,\bar y)$. Let $\Pi$ be an optimal transport plan 
from $f(\bar x)$ to $f(\bar y)$
whose graph $G=(V,E,m,m_0,m_1)$ is a forest.

\begin{lemm}
For all index $i_0$, there is a 
path in $G$ connecting $x_{i_0}$ to $y_{i_0}$
\end{lemm}

\begin{proof}
The choice of $f$ shows that all $m_0(v),m_1(v)$ are integer multiples
of $\alpha 2^{-k}$, so that all $m(e)$ are integer multiples of $\alpha 2^{-k}$.
Let $n(e),n_0(v),n_1(v)\in \mathbb{N}$ be such that $m(e)=n(e) \alpha 2^{-k}$,
$m_0(v)=n_0(v) \alpha 2^{-k}$ and $m_1(v)=n_1(v)\alpha 2^{-k}$. Then the only
$v\in V=\supp f(\bar x)\cup\supp f(\bar y)$ such that $n_0(v)$ contains
$2^{k-i_0}$ in its base-$2$ expansion is $x_{i_0}$. Similarly,
the only $w\in V$ such that $n_1(w)$ contains $2^{k-i_0}$ in its base-$2$ expansion is
$y_{i_0}$. Let $E'\subset E$ be the set of edges $e$ such that $n(e)$ contains
$2^{k-i_0}$ in its base-$2$ expansion.

Any vertex $v$ such that $n_0(v)-n_1(v)$ does not contain $2^{k-i_0}$ in its 
base-$2$ expansion
must be adjacent to an even number of edges of $E'$ due to mass invariance.
Therefore the non-oriented graph induced
by $E'$ has exactly two points of odd degree: $x_{i_0}$ and $y_{i_0}$.
It is well known and a consequence of a simple double-counting argument
that a graph has an even number of odd degree vertices, from which it follows
that the $E'$-connected component of $x_{i_0}$ must contain $y_{i_0}$.
\end{proof}

From now on, fix $i_0$ an index that maximizes $d(x_i,y_i)$
and let $P_0$ be a minimal path between $x_{i_0}$ and $y_{i_0}$. Each
final point of each edge in this path has to be some $y_i$, all distinct by minimality,
so that $P_0$ has length at most
$k$. It follows by a convexity argument that $c(P_0)$
is at least $k(d(x_{i_0},y_{i_0})/k)^p$. Lemma \ref{lemm:forest}
implies $\phi(P)\geqslant \alpha 2^{-k}$ 
so that the cost of $\Pi$ is at least $\alpha 2^{-k} d( x_{i_0},y_{i_0})^p/k^{p-1}$.
We get
\[\wassd_p(f(\bar x),f(\bar y))\geqslant 
  \frac{\alpha^{\frac1p} 2^{-\frac kp}}{k^{1-\frac1p}} d_\infty(\bar x,\bar y)
  \geqslant \frac{1}{k (2^k-1)^{\frac1p}} d_p(\bar x,\bar y)\]
which ends the proof of Theorem \ref{theo:embedding}.

\subsection{Discussion of the embedding constants}\label{sec:constants}

One can wonder if the constants in Theorem \ref{theo:embedding} are optimal.
We shall see in the simplest possible example that they are off by at most 
a polynomial factor, then see how they can be improved in a specific case.

\begin{prop}\label{prop:optimality}
Let $X=\{0,1\}$ where the two elements are at distance $1$ and consider a map
$g : X^k\to\mathscr{W}_p(X)$ such that
\[m \,d_p(\bar x,\bar y)\leqslant \wassd_p(g(\bar x),g(\bar y)) \leqslant M\, d_p(\bar x,\bar y)\]
for all $\bar x, \bar y\in X^k$ and some positive constants $m,M$. Then
\[ m\leqslant \frac{1}{(2^k-1)^{\frac1p}}
  \quad\mbox{and}\quad
  \frac M m \geqslant \left(\frac{2^k-1}{k}\right)^{\frac1p}.\]
Moreover there is a map whose constants
satisfy $m = (2^k-1)^{-\frac1p}$ and $M/m\leqslant(2^k-1)^{\frac1p}$.
\end{prop}

\begin{proof}
By homogeneity, it is sufficient to consider $p=1$, in which case $X^k$
is the $k$-dimensional discrete hypercube endowed with the Hamming
metric: two elements are at a distance equal to the number of bits by which they differ.
Moreover $\mathscr{W}_1(X)$ identifies with the segment $[0,1]$ endowed with the usual
metric $|\cdot|$: a number $t$ corresponds to the measure $t\delta_0+(1-t)\delta_1$.

The diameter of $X^k$ is $k$, so that the diameter of $g(X^k)$ is at most $Mk$.
Since $g(X^k)$ has $2^k$ elements, by the pigeon-hole principle at least two of them
are at distance at most $(2^k-1)^{-1}Mk$. Since the distance between their inverse images is
at least $1$, we get $m\leqslant (2^k-1)^{-1}Mk$ so that $M/m\geqslant(2^k-1)/k$.
The pigeon-hole principle also gives $m\leqslant (2^k-1)^{-1}$ simply by using that
$\mathscr{W}_1(X)$ has diameter $1$.

To get a map $g$ with  $M/m=(2^k-1)$, it suffices to use a Gray code: it is an enumeration
$x_1,x_2,\ldots,x_{2^k}$ of the elements of $X^k$, such that two consecutive
elements are adjacent (see for example \cite{Hamming}).
Letting $f(x_i):=(i-1)/(2^k-1)$ we get a map with $M\leqslant 1$ and  $m=(2^k-1)^{-1}$.
\end{proof}

Note that in Proposition \ref{prop:optimality},
one could improve the lower bound on $M/m$ by a factor
asymptotically of the order of $2^{\frac1p}$ by using the fact that
every element in $X^n$ has an antipode, that is an element at distance $n$ from it.

Let us give an example where the constants are much better.
\begin{exem}\label{exam:ultrametric}
Let $X=\{0,1\}^{\mathbb{N}}$ with the following metric: given
$x=(x^1,x^2,\ldots)\neq y=(y^1,y^2,\ldots)$ in $X$,
$d(x,y)=2^{-i}$ where $i$ is the least index such that $x^i\neq y^i$. 
Then given $k$, let $\ell$ be the least integer such that $2^\ell\geqslant k$
and let $w_1,\ldots,w_k\in \{0,1\}^\ell$ be distinct words on $\ell$
letters. For $x=(x^1,x^2,\ldots)\in X$ and $w=(w^1,\ldots,w^\ell)\in\{0,1\}^\ell$,
define $wx$ as the element $(w^1,w^2,\ldots,w^\ell,x^1,x^2,\ldots)$ of
$X$.

Now let $g:X^k\to\mathscr{W}_p(X)$ be defined by 
\[g\big(\bar x=(x_1,\ldots,x_k)\big) = \sum_{i=1}^k \frac1k \delta_{w_i x_i}.\]
For all $x,y\in X$ and all $i\neq j$, we have $d(w_i x,w_j y)\geqslant 2^{-\ell}\geqslant d(w_i x,w_i y)$.
It follows that 
\begin{eqnarray*}
\wassd_p(g(\bar x), g(\bar y)) &=& \left(\frac1k\sum_i 2^{-p\ell} d^p(x_i,y_i)\right)^{\frac1p} \\
  &=& \frac1{k^{\frac1p}2^{\ell}} d_p(\bar x,\bar y).
\end{eqnarray*}
For this example, we have $M=m$ and moreover $m$ has only the order of $k^{-1-\frac1p}$
instead of being exponentially small. 
\end{exem}

This example could be generalised to more general spaces, for example the
middle-third Cantor set. What is important is that the various components of a given
depth are separated by a distance at least the diameter of the components and that
the metric does not decrease too much between $d(x,y)$ and $d(wx,wy)$ (any bound
that is exponential in the length of $w$ would do).

\subsection{Dynamical largeness}\label{sec:dynamics}

In this section, $X$ is assumed to be compact. Given a continuous
map $\varphi:X\to X$, for any
$n\in\mathbb{N}$ one defines a new metric on $X$ by
\[d_{[n]}(x,y):= \max\{d(\varphi^i(x),\varphi^i(y));0\leqslant i\leqslant n\}.\]
Given $\varepsilon>0$, one says that a subset $S$ of $X$ is
$(n,\varepsilon)$-separated if $d_{[n]}(x,y)\geqslant \varepsilon$ whenever
$x\neq y\in S$. Denoting by $P(\varphi,\varepsilon,n)$ the maximal size of a 
$(n,\varepsilon)$-separated set, the topological entropy of $\varphi$ is defined as
\[\entropy(\varphi) := \lim_{\varepsilon\to 0} \limsup_{n\to+\infty} 
\frac{\log P(\varphi,\varepsilon,n)}{n}.\]
Note that this limit exists since $\limsup_{n\to+\infty} \frac1n \log P(\varphi,\varepsilon,n)$
is nonincreasing in $\varepsilon$.
The adjective ``topological'' is relevant since $\entropy(\varphi)$ does not depend upon the
distance on $X$, but only on the topology it defines.
The topological entropy is in some sense a global measure of the dependance on initial condition
of the considered dynamical system. 
The map $\times d:x\mapsto dx \mod 1$ acting on the circle
is a classical example, whose topological entropy is $\log d$.

Topological entropy was first introduced by Adler, Konheim and 
McAndrew \cite{Adler-Konheim-McAndrew} and the present definition was
given independently by Dinaburg \cite{Dinaburg} and Bowen \cite{Bowen}.

Now, the metric mean dimension is
\[\mdim(\varphi,d) := \liminf_{\varepsilon\to 0} \limsup_{n\to+\infty} 
  \frac{\log P(\varphi,\varepsilon,n)}{n|\log\varepsilon|}.\]
It is zero as soon as topological entropy is finite.
Note that Lindenstrauss and Weiss define the metric mean dimension using
covering sets rather than separated sets; but this does not matter since
their sizes are comparable.

Let us now prove that when $\entropy(\varphi)>0$, then 
$\varphi_\#:\mathscr{W}_p(X)\to\mathscr{W}_p(X)$ has positive 
metric mean dimension. 

\begin{proof}[Proof of Corollary \ref{coro:dynamics}]
Let $\varepsilon,\eta>0$ and $k$ be such that
$\eta\geqslant k(2^k-1)^{\frac1p}\varepsilon$. If $A$ is a $(n,\eta)$-separated
set for $(X,\varphi,d)$ then $A^k\subset X^k$ is a $(n,\eta)$ separated
set for $(X^k,\varphi_k,d_\infty)$. Then Theorem
\ref{theo:embedding} shows that $f(A^k)$ is a $(n,\varepsilon)$-separated
set for $(\mathscr{W}_p(X),\varphi_\#,\wassd_p)$, so that
\[P(\varphi_\#,\varepsilon,n)\geqslant \left(P(\varphi,k(2^k-1)^{1/p}\varepsilon,n)\right)^k.\]

Let $H<\entropy(\varphi)$ and $\beta<1$. For all
$\varepsilon>0$ small enough, and for arbitrarily large integer $n$
we have $P(\varphi,\varepsilon,n)\geqslant \exp(n H)$. Define
\[k=\left\lfloor \frac{\beta p(-\log\varepsilon)}{\log 2}\right\rfloor;\]
then $k(2^k-1)^{1/p}\varepsilon=O\left((-\log\varepsilon)\varepsilon^{1-\beta}\right)\to 0$
when $\varepsilon \to 0$.
Therefore, for all small enough $\varepsilon$, there are arbitrarily large
$n$ such that 
\begin{eqnarray*}
P(\varphi_\#,\varepsilon,n) &\geqslant& \exp(nHk) \\
  &\geqslant& \exp\left(nH\left(\frac{\beta p}{\log 2}(-\log \varepsilon)-1\right)\right) \\
\frac{\log P(\varphi_\#,\varepsilon,n)}{n(-\log \varepsilon)}  
  & \geqslant& \frac{H\beta p}{\log 2} - \frac{H}{-\log\varepsilon}\\
\mdim(\varphi_\#,\wassd_p)&\geqslant& \frac{H\beta p}{\log 2}
\end{eqnarray*}
Letting $H\to \entropy(\varphi)$ and $\beta\to 1$ gives
\[\mdim(\varphi_\#,\wassd_p)\geqslant p\frac{\entropy(\varphi)}{\log 2}\]
as claimed.
\end{proof}

In the case of the shift on certain metrics on
$\{0,1\}^{\mathbb{N}}$, one could
want to use the better bound obtained in Example \ref{exam:ultrametric}.
But the map $g$ defined there does not intertwin $\varphi_k$ and $\varphi_\#$,
and the method above does not apply.

\section{Embedding Hilbert cubes}\label{sec:HC}

In this last section we prove the two theorems about embeddings of Hilbert cubes
in Wasserstein spaces and deduce consequences on their critical parameters. 

\subsection{Embedding small Hilbert cube in the general case}

This section is devoted to the proof of Theorem \ref{theo:HC1}.
We use the same kind of map as in the proof of Theorem \ref{theo:embedding},
but with coefficients that decrease faster to get better point separation.

We assume here that $X$ is compact.
Let $\lambda,\beta\in(0,1)$ be real numbers to be more precisely chosen
afterward and consider the following map:
\begin{eqnarray*}
g:\quad\quad X^{\mathbb{N}} &\to& \mathscr{W}_2(X) \\
  \bar x =(x_1,x_2,\dots) &\mapsto& \frac{1-\beta}\beta\sum_{n=1}^\infty
  \beta^n \delta_{x_n}
\end{eqnarray*}
where $X^{\mathbb{N}}$ will be identified with $\hcube(X;(\lambda^n))$.
We choose $\beta<1/2$, so that $g$ is one-to-one. It is readily seen
to be a continuous map (when $X^{\mathbb{N}}$ is endowed with the
product topology), and we have to bound from below
$\wassd_2(g(\bar x),g(\bar y))$ for all $\bar x,\bar y\in X^{\mathbb{N}}$.

First, since $g(\bar x)$ gives a mass at least $1-\beta$ to $x_1$,
it gives a mass at most $\beta$ to $X\setminus\{x_1\}$, and \emph{any}
transport plan from $g(\bar x)$ to $g(\bar y)$ moves
a mass at least $1-2\beta$ from $x_1$ to $y_1$. 
We already have $\wassd_2(g(\bar x),g(\bar y))^2\geqslant (1-2\beta)d(x_1,y_1)^2$.

If all distances $d(x_n,y_n)$ are of the same order as $d(x_1,y_1)$,
then this first bound is sufficient for our purpose. Otherwise,
we shall reduce to an optimal transport problem involving partial measures.
Define a new map $g_2$ by
$g_2(\bar x)=\frac{1-\beta}\beta\sum_{n=2}^\infty
  \beta^n \delta_{x_n}$; its images are measures of mass $\beta$.
Note that all the theory of optimal transport applies to non probability
measures, as soon as the source and target measures have the same, finite
total mass.
Define a new cost function 
\[\tilde c(x,y) = \min\left(d(x,y)^2,d(x,y_1)^2+d(x_1,y)^2\right)\]

Let $\Pi$ be any transport plan from $g(\bar x)$ to $g(\bar y)$. Then it
can be written $\Pi=\Pi_1+\Pi_\rightarrow+\Pi_\leftarrow+\Pi_\leftrightarrow$
where:
\begin{itemize}
\item $\Pi_1$ has mass between $1-2\beta$ and $1-\beta$
 and is supported on $\{(x_1,y_1)\}$,
\item $\Pi_\rightarrow$ is supported on $\{x_2,x_3,\dots\}\times\{y_1\}$,
\item $\Pi_\leftarrow$ is supported on $\{x_1\}\times\{y_2,y_3,\dots\}$
  and has same mass as $\Pi_\rightarrow$,
\item $\Pi_\leftrightarrow$ is supported on 
  $\{x_2,x_3,\dots\}\times\{y_2,y_3,\dots\}$.
\end{itemize}
To see this, proceed as follows. 
First, letting $h:(n,m)\mapsto (x_n,y_m)$
there is a measure $\Pi'$ on $\mathbb{N}\times\mathbb{N}$
such that $h_\#\Pi'=\Pi$ and the marginals of
$\Pi'$ both are equal to $\frac{1-\beta}\beta\sum \beta^n \delta_n$.
This is a direct application of classical methods, see for example
the gluing lemma in \cite{Villani}.
Then, let $\Pi'_1$, $\Pi'_\leftrightarrow$, $\Pi'_\rightarrow$ and 
$\Pi'_\leftarrow$ be the restrictions of $\Pi'$ to $\{(1,1)\}$,
$\{2,3,\dots\}\times\{2,3,\dots\}$,
$\{2,3,\dots\}\times\{1\}$ and $\{1\}\times\{2,3,\dots\}$.
Then, setting $\Pi_*:=h_\#\Pi'_*$ produces the desired decomposition.

Let $m$ be the mass of $\Pi_\rightarrow$ (which equals the mass of 
$\Pi_\leftarrow$) and define 
\[\Pi_\rightarrow * \Pi_\leftarrow=\frac1m (p_1)_\#(\Pi_\rightarrow)\otimes 
(p_2)_\#(\Pi_\leftarrow)\]
where $p_i$ is the projection on the $i$-th factor. If
we were to identify $x_1$ to $y_1$, this would define a concatenation
of $\Pi_\rightarrow$ and $\Pi_\leftarrow$ (note that the use of a product
is sensible here, since the trajectories to be concatenated all would pass
through $x_1\simeq y_1$ and their is no specific coupling
between the $x_n$'s and the $y_m$'s to remember).

Define further $\tilde\Pi=\Pi_\leftrightarrow+\Pi_\rightarrow * \Pi_\leftarrow$.
It is in some sense the $g_2$ part of $\Pi$, in particular it has mass 
$\beta$.

Let us prove that, denoting
by $c(\Pi)$ the total cost of the transport plan $\Pi$ under the cost
function $c=d^2$, we have
\[c(\Pi)\geqslant c(\Pi_1) + \tilde c(\tilde \Pi)\]
The cost of $\Pi$ is the sum of the cost of its parts, and
the second term of the right-hand side is to bound from below
$c(\Pi_\rightarrow + \Pi_\leftarrow+\Pi_\leftrightarrow)$.
Consider a small amount of mass moved by this partial transport 
plan; it goes under $\Pi$ from some $x_i$  to some $y_j$ ($i,j\geqslant 2$)
either directly, or it is moved to $y_1$ and an equivalent
amount of mass is moved from $x_1$ to $y_j$. In the first case
we use $\tilde c\leqslant c$, in the second case we use
$\tilde c(x,y)\leqslant d(x,y_1)^2+d(x_1,y)^2$.

As already stated, $c(\Pi_1)\geqslant(1-2\beta)d(x_1,y_1)^2$, and
we have left to evaluate $\tilde c(\tilde \Pi)$.
Given $x,y$, set $d_{11}=d(x_1,y_1)$, $a=d(x,y_1)$ and $b=d(x_1,y)$.
By the triangle inequality, $a+b+d_{11}\geqslant d(x,y)$.
Using $a^2+b^2\geqslant \frac12(a+b)^2$, it comes
$\tilde c\geqslant \min(d^2,\frac12 d^2- dd_{11}+\frac12 d_{11}^2)$.
We shall bound $-dd_{11}$ by using
$(\sqrt{\varepsilon} d- d_{11}/\sqrt{\varepsilon} )^2\geqslant 0$ for any
positive $\varepsilon<1$ to be optimized later on. The inequality
\[\tilde c\geqslant \frac12(1-\varepsilon) d^2 
 -\frac12\left(\frac1\varepsilon-1\right) d_{11}^2\]
follows. We therefore get
$c(\Pi)\geqslant A d_{11}^2 + \frac{1-\varepsilon}2 c(\tilde\Pi)$
where
\[A=1-\beta\left(\frac32+\frac1{2\varepsilon}\right)\]
is positive if $\varepsilon$ is large enough (precisely
$\varepsilon>\beta/(2-3\beta)$).
Since $\tilde\Pi$ is a transport plan from $g_2(\bar x)$ to
$g_2(\bar y)$ where $g_2$ is merely $g$ composed with the left shift,
an induction shows that
\[c(\Pi)\geqslant A d(x_1,y_1)^2 +AB d(x_2,y_2)^2 + AB^2 d(x_3,y_3)^2+\dots\]
where
\[B=\frac\beta2(1-\varepsilon)\]
As a consequence, $g$ is sub-Lipschitz (with constant $\sqrt{A/B}$) from
$\hcube(X;(\lambda^n))$ to $\mathscr{W}_2(X)$ where $\lambda=\sqrt{B}$.
The condition on $\varepsilon$ implies that
\[ B<\beta\frac{1-2\beta}{2-3\beta}\]
and any such $B$ with $\beta<1/2$ can be obtained.
The optimal value for $\beta$ is $1/3$, which gives an upper bound
of $1/9$ on $B$. We can therefore get any $\lambda<1/3$.

\subsection{Embedding large Hilbert cube in the rectifiable case}

This section is devoted to the proof of Theorem \ref{theo:HC2}.
The idea is to use the self-similarity at all scales
of the unit cube $I^d$ to
embedd \emph{isometrically} $\hcube(I^d,\bar a)$ inside
$\mathscr{W}_2(I^d)$ for some sequences $\bar a$. The claimed result
will follow immediately, since a direct computation shows that a bi-Lipschitz
embedding $A\to B$ gives a bi-Lipschitz embedding 
$\mathscr{W}_2(A)\to\mathscr{W}_2(B)$ by push forward.

The first step is to find appropriately scaled copy of $I^d$ in itself;
the following is an elementary geometrical fact.
\begin{lemm}\label{lemm:cubes}
Let $\bar c=(c_n)_n$ be an $\ell^d$ sequence of positive numbers.
Then there exist a constant $K$ depending only on $d$ and $\sum c_n^d$
and a family of homotheties $h_n:I^d\to I^d$ 
with disjoint images and ratio $Kc_n$.
\end{lemm}

\begin{proof}
Of course, the existence of such homotheties is equivalent to the existence
of disjoint cubes (oriented according to the coordinate axes)
of sidelength $Kc_n$ in the unit cube $I^d$.
Note that the condition $\bar c\in \ell^d$ is also necessary by volume considerations.
Figure \ref{fig:boxes} illustrates the idea of the proof.

\begin{figure}\begin{center}
\input{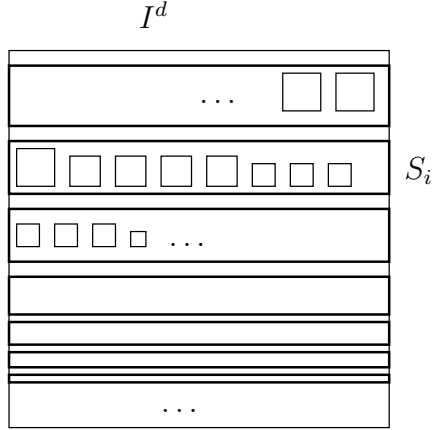}
\caption{After dividing the side-length sequence into blocks,
  we apply the induction hypothesis to each block to get
  adequate families of boxes in each slide $S_i$.}\label{fig:boxes}
\end{center}\end{figure}

Since the result is independent of the order of the terms of $\bar c$,
and since $\lim \bar c$ must be zero, we can assume that $\bar c$
is non-increasing. Up to a dilation 
we can moreover assume that by a factor $\Vert \bar c \Vert_d\leqslant 1$,
in particular $c_n\leqslant 1$ for all $n$.

Define recursively $n_0=0$ and 
\[n_{i+1}=\max\left\{n \,\middle|\, \sum_{k=n_i+1}^n c_n^{d-1}\leqslant 2\right\}\]
It is possible that $n_i=+\infty$ for some $i$; let us momentarily
assume it is not. 

We then have $\sum_{k=n_i+1}^n c_n^{d-1}\geqslant 1$, and
\begin{eqnarray*}
1 &>& \sum_{k=1}^n c_k^d =\sum_{i=0}^{\infty} \sum_{k=n_i+1}^{n_{i+1}} c_k^d\\
        &\geqslant& \sum_{i=0}^{\infty} c_{n_{i+1}} 
                    \sum_{k=n_i+1}^{n_{i+1}} c_k^{d-1}
        \geqslant \sum_{i=0}^{\infty} c_{n_{i+1}}
\end{eqnarray*}
In other words, we have divided
the terms of $\bar c$ into groups of uniformly bounded $\ell^{d-1}$ norm,
in such a way that the sequence of first terms of the groups is $\ell^1$. 
Of course, if $n_i$ is infinite for some $i$, then we have the same conclusion
with one group that is infinite.

Consider inside $I^d$ non-overlapping slices of the form 
$S_i=I^{d-1}\times [a_i,b_i]$ such that $|b_i-a_i|>c_{n_i}$.
By induction on $d$, for all $i$ we can find subcubes of $I^{d-1}$
of side length equal (up to a constant depending only on $d$) to
$c_{n_i+1}, c_{n_i+2},\dots, c_{n_{i+1}}$. We can therefore find
$d$-dimensional cubes in $S_i$ of the same sidelengthes, and we are done. 
\end{proof}

Given an $\ell^d$ positive sequence $\bar c$, and up to taking a smaller 
factor $K$ than given in the Lemma, we can find homotheties $h_n$ of ratios
$Kc_n$ such that the cubes $C_n=h_n(I^d)$ are not only disjoint, but satisfy
the following separation property: for all $x,y\in C_n$ and all $z\in C_m$
with $m\neq n$, $d(x,y)<d(x,z)$.

Let $\bar b=(b_n)$ be any positive $\ell^1$ sequence of sum $1$,
let $a_n=b_n^{1/2} c_n$ and consider the map
\begin{eqnarray*}
h : \hcube(I^d;\bar a) &\to& \mathscr{W}_2(I^d) \\
    \bar x & \mapsto & \sum_{n=1}^\infty b_n \delta_{h_n(x_n)}
\end{eqnarray*}
The separation property on the cubes $C_n$ ensures
that the optimal transport plan from $h(\bar x)$ to
$h(\bar y)$ must be the obvious one, namely
$\Pi=\sum_n b_n \delta_{h_n(x_n)}\otimes\delta_{h_n(y_n)}$.
It has cost $\sum_n b_n K^2c_n^2 d^2(x_n,y_n)=K^2d(\bar x,\bar y)^2$. 

The question is now which sequences $\bar a$ can be decomposed
into a product of an $\ell^2$ and an $\ell^d$ sequence.
If $\bar a \in\ell^{2d/(d+2)}$, one can take $\bar b:=\bar a^{2d/(d+2)}$
and $\bar c=\bar a^{2/(d+2)}$, so that $a_n=b_n^{1/2} c_n$ holds and
the sequences have the right summability properties to apply what
precedes. We have proved the following.
\begin{theo}
If $\bar a$ is any positive $\ell^{2d/(d+2)}$ sequence, there is
a map 
\[h:\hcube(I^d;\bar a) \to \mathscr{W}_2(I^d)\]
that is a homothetic embedding in the sense that
$d(h(\bar x),h(\bar y))=Kd(\bar x,\bar y)$ for some constant $K$.
\end{theo}
Another way to put it is that there is a constant $K$ and an isometric
embedding $\hcube(I^d;K\bar a) \to \mathscr{W}_2(I^d)$.

Note that H\"older's inequality shows that one cannot apply
our strategy to sequences not in $\ell^{2d/(d+2)}$. In fact,
as we shall see below,
the upper bound of Theorem \ref{theo:manifolds} shows that
the exponent $2d/(d+2)$ cannot be improved in general, even for
a mere bi-Lipschitz embedding.

\subsection{Largeness of Wasserstein spaces}\label{sec:largeness}

Let us conclude with the proofs of largeness results claimed in the
introduction.

\begin{proof}[Proof of Proposition \ref{prop:largeness1}]
Let $X$ be a compact metric space of positive Hausdorff dimension
and $\lambda\in (0,1/3)$.
By the embedding theorem \ref{theo:HC1}, we have a continuous
sub-Lipschitz embedding 
$\hcube(X;(\lambda^n))\hookrightarrow \mathscr{W}_2(X)$. Proposition
\ref{prop:expdecay} tells us that 
\[\crit_{\mathscr{I}_2} \hcube(X;(\lambda^n))
  \geqslant \frac{\dim X}{2\log\frac1\lambda}\]
and by Lipschitz monotonicity the same holds  
for $\mathscr{W}_2(X)$. Letting $\lambda$ go to $1/3$ finishes the proof.
\end{proof}

Last the lower and upper bounds in Theorem \ref{theo:manifolds}
can be individually stated under more general hypotheses.

\begin{prop}\label{prop:largeness2}
If $X$ contains a bi-Lipschitz image of a Euclidean cube $I^d$, then
$\mathscr{W}_2(X)$ has at least power-exponential size, and more precisely
\[\crit_{\mathscr{P}}\mathscr{W}_2(X)\geqslant d.\] 
\end{prop}

\begin{proof}
According to Theorem \ref{theo:HC2}, there is a bi-Lipschitz embedding
from $\hcube(I^d;(n^{-\alpha}))$ to $\mathscr{W}_2(X)$ for
all $\alpha>(d+2)/(2d)$. Proposition \ref{prop:polydecay} tells that 
$\hcube(I^d;(n^{-\alpha}))$ has power-exponential critical parameter
bounded below by $2/(2\alpha-1)$, which goes to $d$ when
$\alpha$ approaches $(d+2)/(2d)$. Monotonicity gives the lower bound
for $\mathscr{W}_2(X)$.
\end{proof}

Let us use a counting argument to prove the following.
\begin{prop}\label{prop:upper}
If $X$ is a compact metric space of finite upper Min\-kow\-ski dimension $d$, then
$\mathscr{W}_2(X)$ has at most power-exponential size, and more precisely
\[\crit_{\mathscr{P}}\mathscr{W}_2(X)\leqslant d.\]
\end{prop}

\begin{proof}
Fix some $d'>d$;
by assumption, for all small enough $\varepsilon$
it is possible to cover $X$ by $D=(1/\varepsilon)^{d'}$ balls 
$(B_i)$ of diameter $\varepsilon$. By taking intersections with complements, we can
instead assume that $B_i$'s are disjoint Borel sets of diameters
at most $\varepsilon$. Consider the map
\begin{eqnarray*}
m : \mathscr{W}_2(X) &\to& I^D \\
    \mu &\mapsto& (\mu(B_i))_i
\end{eqnarray*}
and endow $I^D$ with the $\ell^1$ metric.
The map $m$ is not continuous, but whenever $E\subset I^D$ has diameter
at most $\sigma$, we have
\[\diam m^{-1}(E) \leqslant (\diam X)\sqrt{\sigma}+\varepsilon.\]
Indeed, given two measures $\mu$, $\nu$ such that 
$\Vert m(\mu)-m(\nu)\Vert_1\leqslant \sigma$,
we can first move an amount of mass $\sigma$ of $\mu$ (by a distance at most
$\diam X$) to get a measure $\mu'$ that has the same images as $\nu$ under $m$,
then consider any transport plan from $\mu'$ to $\nu$ that is supported
on $\cup B_i\times B_i$ (that is, move mass only inside each $B_i$). This
last transport plan has cost at most $\varepsilon^2$ and the triangular inequality
provides the claimed bound.

Now, for all $D'>D$ and assuming $\varepsilon$ is small enough, it is possible
to cover $I^D$ by $(1/\varepsilon)^{2D'}$ balls $(E_j)$ of diameter
at most $\varepsilon^2$. We get a covering $(\overline{m^{-1}(E_j)})_j$
of $\mathscr{W}_2(X)$ by $(1/\varepsilon)^{2D'}$ sets of diameters at most 
$(\diam X+1)\varepsilon$. Writing $D'=D+\eta/2$, it comes
\[N(\mathscr{W}_2(X),(\diam X+1)\varepsilon) \leqslant
  e^{\left(2\left(\frac1\varepsilon\right)^{d'}+\eta\right)\log\frac1\varepsilon}\]
so that $\ucrit_{\mathscr{P}}\mathscr{W}_2(X) \leqslant d''$
for all $d''>d'>d$, and we are done.
\end{proof}

Now Theorem \ref{theo:manifolds} follows: $X$ being a manifold, it
has upper Minkowski dimension $d$ and contains a bi-Lipschitz image
of $I^d$, so both bounds apply.

\section{Largeness of subsets sets}\label{sec:Hausdorff}

In this section, we briefly explain how to deduce Theorem \ref{theo:Hausdorff}
using the same methods than above.

Let us recall that, when $X$ is a compact metric space, 
$\mathscr{C}(X)$ denotes the set of all closed subsets of $X$,
endowed with the Hausdorff metric.

Generalizing Hilbert cubes, whenever $\bar a=(a_n)_n$
is a sequence of positive reals such that $\lim_n a_n =0$,
let us denote by $\bcube(X,\infty,\bar a)$ the space
$X^{\mathbb{N}}$ endowed with the metric
\[d_{\bar a}(\bar x,\bar y) = \sup_n a_n d(x_n,y_n)\]
Such a space shall be called a \emph{Banach cube} while of
course, topologically it is a Hilbert cube. One can similarly define Banach 
cubes $\bcube(X,p,\bar a)$ for any $p\in[1,\infty]$, but we do not need
that level of generality.
The methods we used to measure the size of Hilbert cubes are easily
generalized to Banach cubes.

\begin{prop}\label{prop:bc}
Let $Y$ be a compact metric space of positive Hausdorff dimension and finite
upper Minkowski dimension.
Then for all positive $\alpha$, it holds
\[\crit_{\mathscr{P}}\bcube(Y,\infty,(n^{-\alpha}))= \frac1\alpha\]
\end{prop}

\begin{proof}
Up to a dilation, we can assume that $Y$ has unit diameter.
Using Frostman's lemma, given any $s\leqslant \dim Y$, there is
a measure $\nu$ on $Y$ and a constant $C$ such that $\nu(B(y,r)\leqslant Cr^s$
for all $y\in Y$ and all $r>0$. Denote by $\mu$ the product measure
$\otimes \nu$ on $B:=\bcube(Y,\infty,(n^{-\alpha}))\simeq Y^{\mathbb{N}}$.
Choose any $\beta>\alpha$ and let $N$ be an integer-valued function 
such that $N(r)\sim r^{-1/\beta}$.
Then for all $\bar y\in B$, and $r>0$ we have
\begin{eqnarray*}
\mu(B(\bar y, r) &=& \prod B(y_n,n^\alpha r) \\
  &\subset& \prod_{n=1}^{N(r)} B(y_n,n^\alpha r) \times Y^{\mathbb{N}}
\end{eqnarray*}
and a quick computation shows that there is a constant $D$ such that
\[\log\mu(B(\bar y, r))\leqslant -D \left(\frac1r\right)^{\frac1\beta}\log\frac1r\]
so that $\crit_{\mathscr{P}} B \geqslant \frac1\beta$. Letting $\beta\to\alpha$,
we have the desired lower bound.

The upper bound is obtained as usual using the upper Minkowski critical parameter.
There is an integer-valued function $M$ such that 
$M(\varepsilon)\sim \varepsilon^{-1/\alpha}$ and $\diam Y_{n^-\alpha}\leqslant \varepsilon$
for all $n\geqslant M(\varepsilon)$. Writting $B=\prod Y_{n^-\alpha}$
and covering each of the terms by $C\varepsilon^{-d}$ balls of diameter
$\varepsilon$, where $C,d$ are constants depending on $Y$, we see that
$B$ can be covered by at most
\[\left(\frac1\varepsilon\right)^{dM(\varepsilon)}\]
balls of diameter $\varepsilon$, and the result follows.
\end{proof}

Now we can deduce the first part of Theorem \ref{theo:Hausdorff}.
\begin{prop}\label{prop:lowerHausdorff}
If $X$ contains a bi-Lipschitz image of a Euclidean cube $I^d$, then
$\mathscr{C}(X)$ has at least power-exponential size, and more precisely
\[\crit_{\mathscr{P}}\mathscr{C}(X)\geqslant d.\] 
\end{prop}

\begin{proof}
Using Lemma \ref{lemm:cubes}, for all $d'>d$
there are homotheties $(h_n)_{n\in\mathbb{N}}$ of ratio
$C n^{-d'}$ from $I^d$ to $I^d$, with $h_n(I^d)$ and
$h_m(I^d)$ separated by a distance at least $C n^{-d'}$ for all $n,m$.
Then the map
\begin{eqnarray*}
\bcube(I^d,\infty,(n^{-d'})) &\to& \mathscr{C}(I^d) \\
(x_1,x_2,\dots) &\mapsto& \overline{\{h_n(x_n) \,|\, n\in\mathbb{N}\}}
\end{eqnarray*}
defines an homothetic embedding of ratio $C$.

Proposition \ref{prop:bc} and the monotonicity property gives the result.
\end{proof}

Finally, the following results ends the proof of Theorem \ref{theo:Hausdorff}.
It is proved just like its Wasserstein analogue.
\begin{prop}\label{prop:upperHausdorff}
If $X$ is a compact metric space of finite upper Min\-kow\-ski dimension $d$, then
$\mathscr{C}(X)$ has at most power-exponential size, and more precisely
\[\crit_{\mathscr{P}}\mathscr{C}(X)\leqslant d.\]
\end{prop}

\begin{proof}
Fix some $d'>d$;
for all small enough $\varepsilon$
it is possible to cover $X$ by $D=(1/\varepsilon)^{d'}$ disjoint Borel sets 
$(B_i)$ of diameter at most $\varepsilon$.
The map
\begin{eqnarray*}
m : \mathscr{C}(X) &\to& \{0,1\}^D \\
   A &\mapsto& (m_i(A))_i
\end{eqnarray*}
defined by $m_i(A)=1$ if and only if $A\cap B_i \neq \emptyset$
has the property that every point in $\{0,1\}^D$ has an inverse image
of diameter at most $\varepsilon$.

We get a covering 
of $\mathscr{C}(X)$ by $2^D$ sets of diameters at most 
$\varepsilon$, and it follows
\[\ucrit_{\mathscr{P}}\mathscr{C}(X)\leqslant d'.\]
This being valid for all $d'>d$, we get the desired result.
\end{proof}

\bibliographystyle{smfalpha}
\bibliography{biblio.bib}

\enlargethispage{2\baselineskip}

\end{document}